\documentclass[12pt]{amsart}
\usepackage{amsmath, amsthm, amscd, amsfonts}

\newtheorem{theorem}{Theorem}[section]
\newtheorem{lemma}[theorem]{Lemma}

\newtheorem{corollary}[theorem]{Corollary}
\theoremstyle{definition}
\newtheorem{definition}[theorem]{Definition}

\theoremstyle{remark}

\newtheorem{remark}[theorem]{Remark}
\numberwithin{equation}{section}

\newfont{\kh}{msbm10}

\begin{document}
\title[A Riesz-Fredholm type theorem]
{A Riesz-Fredholm type theorem on certain Hilbert C*-modules}

\author{Z. Panahi}
\address{Zahra Panahi, \newline Faculty of Mathematical Sciences,
Shahrood University of Technology, P. O. Box 3619995161-316,
Shahrood, Iran } \email{zpanahi654@gmail.com}
\author{K. Sharifi}
\address{Kamran Sharifi, \newline Faculty of Mathematical Sciences,
Shahrood University of Technology, P. O. Box 3619995161-316,
Shahrood, Iran} \email{sharifi.kamran@gmail.com}


\begin{abstract}
Let $C$ be compact modular operator on a Hilbert C*-module $E$ satisfying property $\mathbb{[H]}$
[{\it J. Math. Phys.} {\bf 49} (2008), 033519], and let $ L :=I-C$.
We prove the existence of a unique natural number $r$ for which $L^r$ is an EP operator on $E$.
Moreover, we show that the kernel of $L^r$ is a finitely generated submodule of $E$ and that $E$ admits the decomposition
$E=Ker(L^r) \oplus Ran(L^r)$.  These results provide a framework
for analyzing the solvability of the equation $x-Cx=f$ on $E$.

\end{abstract}

\maketitle
%_____________________________________________________________________

\section{Introduction and preliminaries.}

In Banach spaces theory, compact operators are linear operators on Banach spaces that map
bounded sets to relatively compact sets.
In the case of a Hilbert C*-modules, the compact
operators are the closure of the finite rank operators in the uniform operator topology.
In general, operators on Hilbert modules feature properties that do not appear
in the Hilbert (or Banach) spaces case, cf. \cite{Arambasic, FR3, LAN, SharifiInv}.
In this paper, we are going to present the basic theory for an operator equation
\begin{equation}\label{RF1}
x-Cx=f
\end{equation}
with a compact modular operator $C: E \to E$ mapping a
Hilbert C*-module $E$ into itself. This work extends the classical framework
established by Riesz and Fredholm for integral equations in Hilbert and Banach spaces, cf. \cite{Fredholm,Riesz}.

A \textit{pre-Hilbert module} over a C*-algebra $ \mathcal{A}$
is a complex linear space $E$ which is an algebraic right
$ \mathcal{A}$-module equipped with an $ \mathcal{A}$-valued inner product
$\langle \cdot, \cdot \rangle:  E \times  E \to  \mathcal{A} $ satisfying
the following properties for each
$x, y, z \in  E$, $\lambda \in \mathbb{C}$, and $a \in \mathcal{A}$:
\begin{enumerate}
   \item $ \langle x, x \rangle \geq 0$, and $\langle x, x \rangle = 0$  if and only if $x = 0,$
   \item $\langle x, y + \lambda z \rangle = \langle x, y \rangle + \lambda \langle x, z \rangle,$
   \item $\langle x, ya \rangle = \langle x, y \rangle a, $
   \item $ \langle y, x \rangle = \langle x, y \rangle^*$.
\end{enumerate}
The $\mathcal{A}$-module  $E$ is called a Hilbert C*-module
if $E$ is complete with respect to the norm $ \Vert x\Vert  : =\Vert \langle
 x,x\rangle \Vert ^{1/2}.$
If $F$ is a (possibly non-closed) $\mathcal{A}$-submodule
of $E$, then $F^\bot :=\{ y \in E: ~ \langle x,y \rangle=0,~ \ {\rm for} \ {\rm
all}\ x \in F \} $ is a closed $\mathcal A$-submodule of $E$ and
$ \overline{F} \subseteq F^{ \perp \, \perp}$.
A closed submodule $F$ of a Hilbert $ \mathcal{A}$-module $E$ is \textit{ topologically
complementable} if there is a closed submodule $G$ such that $E=F + G$ and $F \cap G = \{0\}$. In this case,
we write $E=F \dot{+} G$.
The submodule $F$  is called \textit{ orthogonally complemented} if we  further have $ F \perp G$.
In this case $G=F^{ \perp}$, and we write $E=F \oplus G$.

Let $E$ be a Hilbert $\mathcal{A}$-module over a C*-algebra $\mathcal{A}$.
We denote by $\mathcal{L}(E)$ %(or alternatively $ \mathcal{L}_{ \mathcal{A}}(E)$ to emphasize the underlying C*-algebra)
the C*-algebra of all adjointable operators on $E$. These are the $\mathcal{A}$-linear maps $T:E \to E$ for which there exists
an adjoint operator $T^*:E \to E$ satisfying $$
\langle Tx,y \rangle =\langle x,T^*y \rangle~ \textrm{for~all}~  x,y \in
E.$$
For any $ x,y \in
E$, the rank one operator $\theta_{x,y}$ is defined by
$$\theta_{x,y}(z) = x \langle y,z \rangle ~ \textrm{for}~ z \in E  .$$
Each $\theta_{x,y}$ is a bounded adjointable operator with adjoint $(\theta_{x,y} )^*=\theta_{y,x}$.
The space of compact operators $ \mathcal{K}(E)$ is defined as the
norm closure of the linear span of all such rank one operators in $\mathcal{L}(E)$.
%We may sometimes use the notation
%$ \mathcal{K}_{ \mathcal{A}}(E)$ in place of $ \mathcal{K}(E)$, to make explicit the underlying C*-algebra.

We note that some familiar properties of Hilbert spaces, such as Pythagoras' equality, self-duality,
the polar decomposition of operators, and even the decomposition
into orthogonal complements, do not hold in the framework of Hilbertian modules.
In general, operators on Hilbert modules feature properties that do not appear
in the Hilbert (or Banach) spaces case.
The reader is encouraged to study the publications \cite{BakicFrame, ELMXZ, FR3, FrankReg, FS2, NHA, ZamaniCmodule}
and the references therein for more detailed information on modular operators and some recent developments
in the theory of Hilbert C*-modules.

We use the notations
$Ker(.)$ and $Ran(.)$ to denote the kernel and range of operators, respectively.
The modular operator $T$ has closed range if and only if $T^*$ has closed range.
In this case, the space $E$ decomposes as $E=Ker(T)
\oplus Ran(T^*)=Ker(T^*) \oplus Ran(T)$; see \cite[Theorem 3.2]{LAN}.

Every bounded sequence in a Hilbert space has a weakly convergent subsequence.
Furthermore, compact operators map weakly convergent sequences to strongly convergent ones.
Now, suppose $E$ is a Hilbert $ \mathcal{A}$-module. One can ask whether for every bounded sequence $(\zeta_{n})$ in $E$,
there exist a subsequence $(\zeta_{n_k})$ and an element $\zeta$ in $E$ such that
\begin{equation*}
\| C \zeta_{n_k} - C \zeta\| \to 0, ~~~~ \mathrm{for~every~} C\in \mathcal{K}(E).
\end{equation*}
In general, the answer is negative. For example, let $ \mathcal{A} = \mathcal{ B}(H)$
for some infinite-dimensional Hilbert space $H$, and consider $E= \mathcal{A}$ as a Hilbert C*-module over itself.
Then the identity operator on $H$, which is not a compact operator on $H$, corresponds to a compact operator in $\mathcal{K}(E)$.
However, since $H$ is infinite dimensional, the identity operator cannot map a
weakly convergent sequence to a strongly convergent one, demonstrating that the desired property fails.
This argument motivates the study of property $\mathbb{[H]}$, initially introduced in \cite{CIMS,FR3},
which will be discussed in more detail in the next section. It is known that every Hilbert C*-module over
a finite dimensional C*-algebra satisfies property $\mathbb{[H]}$.

Let $E$ be a Hilbert C*-module over a unital C*-algebra $ \mathcal{A}$ that satisfies
property $\mathbb{[H]}$ and let $L=I-C$, where $C \in \mathcal{K}(E)$ is a compact operator.
In this paper, we study the equation $x-Cx=f$ and prove that
there exists a unique natural number $r$ such that  $L^r$ is an EP operator on $E$, that is, the range
of $L^r$ is closed and satisfies $Ran(L^r)=Ran(L^{r*})$.
In particular, the kernel of $L^r$ is a finitely generated submodule of $E$, and $E$
admits the direct sum decomposition $E=Ker(L^r) \oplus Ran(L^r)$.
Some further useful results are also obtained.

\section{Hilbert C*-modules satisfying property $\mathbb{[H]}$.}

In this section, we review the property $\mathbb{[H]}$ introduced by
Chmieli\'nski, Ili\v{s}evi\'c, Moslehian and Sadeghi  \cite{CIMS}, along with a characterization of finite
dimensional C*-algebras established by
Aramba\v{s}i\'c, Baki\'c and Raji\'c \cite {Arambasic}. Definition \ref{CIMS},
Lemma \ref{Arambasic1} and Lemma \ref{Riesz1} serve as foundational elements for our main results.

\begin{definition} \label{CIMS} (\cite[Proposition 2.1]{CIMS})
Let $ \mathcal{A}$ be a C*-algebra. The right Hilbert $\mathcal{A}$-module $E$ satisfies property $\mathbb{[H]}$,
if for every bounded sequence
$(\zeta_n)$ in $E$ there are a subsequence $(\zeta_{n_k})$ of $(\zeta_n)$ and $\zeta \in E$ such that
\begin{equation}\label{CIMS1}
\langle v , \zeta_{n_k} \rangle \to \langle v , \zeta \rangle, ~~~~ \mathrm{for~every~} v \in E.
\end{equation}
\end{definition}
Frank has shown that every
Hilbert C*-module over an arbitrary finite dimensional C*-algebra of coefficients is a
self-dual Hilbert module, cf. \cite[Proposition 4.4]{FR3}. Chmieli\'nski et al. utilized this result to establish the following lemma:
\begin{lemma} \label{CIMS2} (\cite[Proposition 2.1]{CIMS})
Let $E$ be a Hilbert C*-module over a finite dimensional C*-algebra $\mathcal{A}$, then
$E$ satisfies property $\mathbb{[H]}$.
\end{lemma}
It is known that full Hilbert C*-modules can
be described by the property $\mathbb{[H]}$ as follows.
Let $E$ be a full right Hilbert $\mathcal{A}$-module.
For every bounded sequence
$(\zeta_n)$ in $E$ there are a subsequence $(\zeta_{n_k})$ of $(\zeta_n)$ and $\zeta \in E$ such that
(\ref{Aram0}) holds if and only if $\mathcal{A}$ is a finite dimensional C*-algebra, cf. \cite[Theorem 2.3]{Arambasic}.
\begin{lemma} \label{Arambasic1} (\cite[Theorem 2.3]{Arambasic})
Let $E$ be a Hilbert $\mathcal{A}$-modules that satisfies property $\mathbb{[H]}$,  then
\begin{equation}\label{Aram0}
\| C \zeta_{n_k} - C \zeta\| \to 0, ~~~~ \mathrm{for~every~} C\in \mathcal{K}(E).
\end{equation}
\end{lemma}
\begin{proof}
To derive (\ref{Aram0}), observe that $\mathcal{K}(E)$ equals the closed linear span of $\{
\theta_{u,v}: u,v \in E \}$ in $ \mathcal{L}(E) $ and
$$ \| \theta_{u,v} (\zeta_{n_k}) - \theta_{u,v} (\zeta) \| =
\| u \langle v , \zeta_{n_k} \rangle - u \langle v , \zeta \rangle \| \to 0,$$
for~every $u,v\in E$.
\end{proof}

The compact operators on Hilbert modules over a finite dimensional C*-algebra are notable
because they exhibit as much similarity to matrices as one can expect from a general operator.

Some aspects of weakly convergent bounded sequences in Hilbert C*-modules, particularly
those related to fundamental properties of compact operators on such modules over an
arbitrary finite dimensional C*-algebra, can be found in \cite{SharifiInv}.

\begin{lemma} \label{Riesz1}
Let $E$ be a Hilbert $ \mathcal{ A}$-module satisfying property $\mathbb{[H]}$, and
let $L:=I-C$, where $C \in \mathcal{K}(E)$. Then the following assertions hold.
\begin{enumerate}
\item For every $y$ in the closure of $Ran(L)$,
there exist sequences $(x_n)$ in $E$ and $(u_n)$ in the closed submodule $Ker(L)$ of $E$ such that
$Lx_n=x_n - Cx_n$ converges to $y$ and
\begin{equation*}\label{Riesz2}
\| x_n - u_n \|= \delta_n:= \mathrm{ inf} \{ \| x_n - u \| \, : \ u\in Ker(L) \}.
\end{equation*}
  \item The sequence $ \zeta_n := x_n - u_n$ is a bounded sequence in $E$.
\end{enumerate}
\end{lemma}
\begin{proof}
Let $y$ in the closure of $Ran(L)$, then there exists a sequence $(x_n)$ in $E$ such that $Lx_n=x_n - Cx_n$
converges to $y$. We establish the existence of a sequence $(u_n)$ in the closed submodule $Ker(L) \subset E$ satisfying
$\| x_n - u_n \|= \delta_n.$
For each fixed $n$, there exists a sequence $(w_m^n)$ in $Ker(L)$ such that
\begin{equation}\label{w0}
\lim_{m} \|x_n - w_m^n \| = \delta_n.
\end{equation}
Consequently, the sequence $(w_m^n)$ is bounded with $\delta_n + 1 + \|x_n \|$, for sufficiently large $m$.
By Lemma \ref{Arambasic1}, we may extract a subsequence $(w_{m_k}^n)$ converging weakly to $u_n \in Ker(L)$ in the sense that
\begin{equation*}\label{w1}
\langle v, w_{m_k}^n \rangle \to \langle v, u_n \rangle \quad \text{for all } v \in E \text{ and } n \in \mathbb{N},
\end{equation*}
with the additional property
\begin{equation*}\label{w2}
\|C w_{m_k}^n - C u_n \| \to 0 \quad \text{for each } n \in \mathbb{N}.
\end{equation*}
Since both $(u_n)$ and $(w_m^n)$ lie in $Ker(L)$, we have the fixed point properties:
\begin{align*}
C u_n &= u_n, \\
C w_{m_k}^n &= w_{m_k}^n,
\end{align*}
which imply,
\begin{equation}\label{w3}
| \, \|  x_n -  u_n \| - \|   x_n -  w_{m_k}^{n} \|   \, | \leq \|  w_{m_k}^{n} -  u_n \|=\| C w_{m_k}^{n} - C u_n \| \to 0.
\end{equation}
Combining \eqref{w0} and \eqref{w3} yields the equality
$ \|  x_n -  u_n \|= \lim_k \|  x_n -  w_{m_k}^{n}  \|=\delta_n,$ which establishes assertion (1).

We claim $ \zeta_n = x_n - u_n$  is a bounded sequence in $E$. Suppose otherwise, there is a
subsequence $( \zeta_{n_k})$ such that $\| \zeta_{n_k} \| \geq k$.
By compactness of $C$ and Lemma \ref{Arambasic1}, there exists
a subsequence $(v_{k_j})$ of $v_k := \frac{ \zeta_{n_k} }{ \| \zeta_{n_k} \|}$ and $v$ such that
$ Cv_{k_j} \to v$.  Since $L \zeta_{n_k}=Lx_{n_k}-Lu_{n_k}=Lx_{n_k} \to y$,
$$ \| Lv_k \|= \frac{ \| L \zeta_{n_k} \| }{ \| \zeta_{n_k} \| } \leq \frac{ \| L \zeta_{n_k} \| }{ k } \to 0, ~~\mathrm{as}~ k \to \infty.$$
We therefore have $v_{k_j}=Lv_{k_j} + K v_{k_j} \to v$, which implies $0= \lim L v_{k_j} = Lv$. Since
$ u_{n_k} + \| \zeta_{n_k} \|v$ is in $Ker(L)$ we have
$$ \| v_k - v \|= \frac{1}{ \| \zeta_{n_k} \| } \, \| x_{n_k} - (u_{n_k} + \| \zeta_{n_k} \|v) \| \geq \frac{1}{ \| \zeta_{n_k} \| } \,
\inf \| x_{n_k} - u \| = 1,$$
which is a contradiction with $ v_{k_j} \to v$. Consequently, the sequence $( \zeta_{n} )$ is bounded, which completes the proof of (2).
\end{proof}

\section{A Riesz-Fredholm type theorem on Hilbert
C*-modules satisfying property $\mathbb{[H]}$. }

The following lemma follows from Riesz's lemma for Banach spaces, which can be
found in Kreyszig's book, see Lemma 2.5-4 in \cite{Kreyszig}.

\begin{lemma} \label{Riesz0}
Let $E$ be a Hilbert C*-module over an arbitrary C*-algebra $ \mathcal{A}$, $0 < \epsilon < 1$, and let $F$
be a closed proper submodule of $E$. Then there exists
$x \in E$ such that $\| x \|=1$ and
$$ \| x - F \| \geq \epsilon,$$
where $$\| x - F \|:=  \mathrm{ inf} \{ \| x - u \| \, : \ u\in F \}.$$
\end{lemma}

\begin{lemma} \label{Conwayp215}
Let $E$ be a Hilbert $ \mathcal{ A}$-module satisfying property $\mathbb{[H]}$ and $C \in \mathcal{K}(E)$.
Then $Ran( \lambda I-C)$ is an orthogonal summand for every nonzero $ \lambda \in \mathbb{C}$. In particular,
$E=Ker (I-C^*) \oplus Ran(I-C)$.
\end{lemma}

\begin{proof}
Let $y$ in the closure of $Ran(L)$ and $L=I-C$. In view of Lemma \ref{Riesz1},
there exist sequences $(x_n)$ in $E$ and $(u_n)$ in the closed submodule $Ker(L)$ of $E$ such that
$Lx_n=x_n - Cx_n$ converges to $y$, $ \zeta_n := x_n - u_n$ is a bounded sequence in $E$ and
$$ \| x_n - u_n \|=  \mathrm{ inf} \{ \| x_n - u \| \, : \ u\in Ker(L) \}.$$
Utilizing Lemma \ref{Arambasic1}, the bounded sequence $( \zeta_{n})$
has a subsequence
$( \zeta_{n_k})$ such that $ C \zeta_{n_k}$ is convergent. Hence, $ \zeta_{n_k}= L \zeta_{n_k} + C \zeta_{n_k}$ converges
to an element
$ \zeta \in E$. On the other hand, $Lx_n \to y$ and $Lx_{n_k}-Lu_{n_k} = L \zeta_{n_k} \to L \zeta $,
it follows that $y=L \zeta$, i.e., the range of $L$ is closed.
Consequently, the closed submodule $Ran(L)$ is an orthogonal summand by \cite[Theorem 3.2]{LAN}.
The range of $\lambda I-C$ is an orthogonal summand, too.
\end{proof}

\begin{lemma} \label{Wegge257}
Let $E$ be a Hilbert C*-module over a unital C*-algebra $ \mathcal{A}$. Then $E$ is
finitely generated if and only if $\mathcal{K}(E)$ is unital.
\end{lemma}
The proof of the preceding lemma follows from the equivalence of statements (6) and (7), as
established in Theorem 15.4.2 and Remark 15.4.3 of Wegge-Olsen's book \cite{WEG}.
A key observation is that every finite dimensional C*-algebra $ \mathcal{A}$ is unital, which follows
from its structure as a direct sum of full matrix algebras, each of which is unital.

In the following results, we unify key matrix-theoretic properties and extend them to
characterize the spectral behavior of compact operators on Hilbert C*-modules.
%%%%%%%%%%%%%%%%%%%%%%%%%%%%%%%%%%%%%%%%%%%%%%%%%%%%%%%%%%%%%%%%%%%%%%%%%%%%%%%%%%%%%%%%%%%%%%%%%%%%
\begin{remark} \label{Conwayp215.1}
If $E$ is a Hilbert $ \mathcal{ A}$-module that is not finitely generated and $C \in \mathcal{K}(E)$ is compact modular operator,
then $0$ belongs to the spectrum of $C$, denoted by
$$ \sigma (C) :=  \{ \lambda \in \mathbb{C} :~ \lambda I -C {\rm ~is ~not ~invertible~in}~ \mathcal{L}(E) \}.$$
This fact follows immediately from Lemma \ref{Wegge257}, see also \cite[Lemma 2.5]{SharifiInv}.
\end{remark}
%\begin{proposition} \label{Conwayp215.2}
%Let $E$ be a Hilbert C*-module over a unital C*-algebra $ \mathcal{A}$ that satisfies property $\mathbb{[H]}$ and $C \in \mathcal{K}(E)$.
%Then every nonzero $\lambda \in \sigma (C)$ is an eigenvalue of $C$, that is, $Ker( \lambda I -C ) \neq \{ 0 \}$.
%\end{proposition}

\begin{lemma} \label{RieszK1}
Let $E$ be a Hilbert C*-module over a unital C*-algebra $ \mathcal{A}$ and
$C \in \mathcal{K}(E)$. Then $Ker(I-C)=\{x \in E:~Cx=x \}$ is a finitely generated Hilbert
$ \mathcal{ A}$-module.
\end{lemma}
\begin{proof} The kernel of the bounded linear operator $L=I-C$ is a closed submodule of $E$.
Since $Cx=x$, for every $x$ in the kernel of $L$, the restriction of $C$ to
$Ker(I-C)$ coincides with the identity operator, that is, $I=C_{|_{Ker(L)}}: Ker(L) \to Ker(L)$
is a compact modular operator from $Ker(L)$ onto $Ker(L)$. Hence, the closed $ \mathcal{A}$-submodule
$Ker(L)$ is a finitely generated Hilbert $ \mathcal{A}$-module by Lemma \ref{Wegge257}.
\end{proof}

\begin{theorem} \label{RieszK2}
Let $E$ be a Hilbert C*-module over a unital C*-algebra $ \mathcal{A}$ that satisfies
property $\mathbb{[H]}$ and $C \in \mathcal{K}(E)$. If $L=I-C$, then
there exists a uniquely determined non-negative integer $r$ such that
\begin{equation}\label{Riesz3}
\{0 \}= Ker(L^0) \subset Ker(L^1) \subset \cdots \subset Ker(L^r) = Ker(L^{r+1})= \cdots ,
\end{equation}
and
\begin{equation}\label{Riesz4}
E= Ran(L^0) \supset Ran(L^1) \supset \cdots \supset Ran(L^r) = Ran(L^{r+1})= \cdots .
\end{equation}
\end{theorem}

\begin{proof}
Since for each $x$ with $Lx=0$ it follows that $L^{n+1}x=0$, we trivially have
$$ \{0 \}= Ker(L^0) \subset Ker(L^1) \subset Ker(L^2) \subset \cdots.$$
Suppose, for contradiction, we assume that  $Ker(L^n) \neq Ker(L^{n+1})$, for every natural number $n$.
We apply Lemma \ref{Riesz0} for the finitely generated submodule $Ker(L^n)$, there exists a sequence
$(x_n)$ in $Ker(L^{n+1})$ such that $\|x_n \|=1$ and
\begin{equation}\label{Riesz5}
\|x_n - x \|\geq \|x_n - Ker(L^n) \| \geq \frac{1}{2},
\end{equation}
for every $x \in Ker(L^n)$. For every $n > m$ we have
\begin{equation*}
Cx_n - Cx_m= x_n - (x_m + Lx_n - Lx_m),
\end{equation*}
and
\begin{equation*}
L^n(x_m+Lx_n-Lx_m)=L^{n-m-1} (L^{m+1} x_m)+L^{n+1}x_n -L^{n-m}(L^{m+1}x_m)=0,
\end{equation*}
which imply  $x_m+Lx_n-Lx_m \in Ker(L^n)$. In view of (\ref{Riesz5}), we have
\begin{equation*}
\|Cx_n - Cx_m \|= \| x_n - (x_m + Lx_n - Lx_m)\| \geq \frac{1}{2},~{\rm for~ all ~ } n >m.
\end{equation*}
Consequently, the sequence $Cx_n$ does not contain a convergent subsequence
which is a contradiction to the compactness of $C$ and the property $\mathbb{[H]}$.
Thus, we now conclude that in the sequence $Ker(L^n)$, there exist two consecutive kernels that are equal.
Suppose $r$ is the smallest natural number $m$ such that the equality $ Ker(L^m) = Ker(L^{m+1})$ holds, then
$$Ker(L^r) = Ker(L^{r+1})=  Ker(L^{r+2}) = \cdots.$$ To see this, let $x \in  Ker(L^{r+2})$ then $L^{r+1} (Lx)=0$,
and so $Lx \in  Ker(L^{r+1})=Ker(L^r)$. We can now conclude inductively that
$ Ker(L^n) = Ker(L^{n+1})$ for every $n\geq r$.

To see the second assertion,  let $F_n=Ran(L^n)$. Then the $F_n$'s are closed $ \mathcal{A}$-submodules by
Lemma \ref{Conwayp215} that satisfy $F_{n+1} \subseteq F_n$ for every nonnegative integer $n$.
Suppose otherwise, we assume that  $F_{n+1} \neq F_n$, for every natural number $n$.
We apply Lemma \ref{Riesz0} for the closed submodule $F_{n+1}$, there exists a sequence
$(y_n)$ in $F_n$ such that $\|y_n \|=1$ and
\begin{equation}\label{Riesz6}
\|y_n - y \|\geq \|y_n - F_{n+1} \| \geq \frac{1}{2},
\end{equation}
for every $y \in F_{n+1}$. For every $m > n$ we have
\begin{equation*}
Cy_n - Cy_m= y_n - (y_m + Ly_n - Ly_m),
\end{equation*}
and
\begin{equation}\label{Riesz7}
y_m + Ly_n - Ly_m=L^{n+1}(L^{m-n-1}x_m+x_n-L^{m-n}x_m),
\end{equation}
where $y_n=L^n x_n$. The equality (\ref{Riesz7}) implies that $y_m + Ly_n - Ly_m \in F_{n+1}$ and so
\begin{equation*}
\|Cy_n - Cy_m \|= \| y_n - (y_m + Ly_n - Ly_m) \| \geq \frac{1}{2},~{\rm for~ all ~ } m >n,
\end{equation*}
which derives a contradiction again to the compactness of $C$ and the property $\mathbb{[H]}$.
Hence, we now deduce that in the sequence $F_n$, there exist two consecutive ranges that are equal.
Suppose $s$ is the smallest natural number $m$ such that the equality $Ran(L^m) = Ran(L^{m+1})$ holds, then
$$Ran(L^s) = Ran(L^{s+1}) = Ran(L^{s+2}) = \cdots .$$
To see this, let $y=L^{s+1}x \in  Ran(L^{s+1})$ then $L^{s}x=L^{s+1} x_0$
for some $x_0 \in E$, since $Ran(L^s) = Ran(L^{s+1})$.
Consequently,  $y=L^{s+2}x_0 \in  Ran(L^{s+2})$ and so $Ran(L^{s+1}) \subseteq Ran(L^{s+2})$.
We can now conclude inductively that
$ Ran(L^n) = Ran(L^{n+1})$ for every $n\geq s$.

We claim that $r=s$. Let $r>s$ and $x \in Ker(L^r)$, then $L^{r-1}x \in Ran(L^{r-1})=Ran(L^r)$ and so
$0=L^r x= L(L^{r-1}x)= L(L^r x_0)=L^{r+1}x_0$. Thus, $x_0 \in Ker(L^{r+1})=Ker(L^r)$, which implies
$L^{r-1}x=L^r x=0$. Consequently, $ Ker(L^r) = Ker(L^{r-1})$, which contradicts the choice of $r$. We now
assume that $r < s$ and $y=L^{s-1} x \in Ran(L^{s-1})$. Then $Ly=L^{s} x \in Ran(L^{s})=Ran(L^{s+1})$, we therefore have
$Ly= L^{s+1}x_0$ and
$$L^s(x-Lx_0)=Ly - L^{s+1}x_0=0,$$ for some $x_0 \in E$.
Since $Ker(L^{s-1})= Ker(L^s)$, we have $L^{s-1}(x-Lx_0)=0$ and so $y=L^sx_0 \in Ran(L^s)$.
Consequently, $ Ran(L^s) = Ran(L^{s-1})$, which contradicts the definition of $s$.

\end{proof}

Let $E$ be a Hilbert
$ \mathcal{A}$-modules and let $T \in \mathcal{L}(E)$. Then a bounded adjointable operator
$T^{ \dag} \in \mathcal{L}(E)$ is called the {\it Moore-Penrose
inverse} of $T$ if
\begin{equation*}
T \, T^{ \dag}T=T, \ T^{ \dag}T \, T^{ \dag}= T^{ \dag}, \ (T \,
T^{ \dag})^*=T \, T^{ \dag} \ {\rm and} \ ( T^{ \dag} T)^*= T^{
\dag} T.
\end{equation*}
Xu and Sheng in \cite{Xu/Sheng} have shown that a bounded
adjointable operator between two Hilbert C*-modules admits a
bounded Moore-Penrose inverse if and only if the operator has
closed range. The reader should be aware of the fact that a
bounded adjointable operator may admit an unbounded operator as
its Moore-Penrose, see \cite{FS2, SHA/Groetsch}.

An operator $T \in \mathcal{L}(E)$ is
called {\it EP} if $Ran(T)$ and $Ran(T^*)$ have the same closure.
Let $T \in \mathcal{L}(E)$ be an EP operator with
closed range,  then $T$ commutes with the bounded operator $T^{ \dag}$;
see \cite{SHA/EP, SharifiErr} for details and some useful examples.
%Products of EP operators on Hilbert C*-modules and some useful examples can be found in \cite{SHA/EP, SharifiErr}.

\begin{theorem} \label{RieszK3}
Let $E$ be a Hilbert C*-module over a unital C*-algebra $ \mathcal{A}$ that satisfies
property $\mathbb{[H]}$ and $C \in \mathcal{K}(E)$. If $L=I-C$, then
there exists a uniquely determined non-negative integer $r$ such that
\begin{equation}\label{Riesz8}
E=Ker(L^r) \oplus Ran(L^r).
\end{equation}
In particular, there exists a uniquely determined non-negative integer $r$ such that
$L^r$ is an EP operator.
\end{theorem}
\begin{proof}
We adhere to the notations and follow the reasoning used in the proof of the previous theorem. Suppose
$y \in Ker(L^r) \cap Ran(L^r)$. Then there exists $x \in E$ such that $y=L^r x$ and $L^{2r}x=0$, that is,
$x \in Ker(L^2r)=  Ker(L^r)$ and so $y=L^r x=0$. We now suppose $x$ is an arbitrary element in $E$. Then
$L^r x \in Ran(L^r)=Ran(L^{2r})$ and so there exists $x_0$ such that $L^rx=L^{2r}x_0$. We set $y=L^{r}x_0$
and $z=x-y$, then $L^rz=L^rx- L^{2r}x_0=0$, which means $x=z+y \in Ker(L^r)  \dot{+} Ran(L^r)$. Hence,
\begin{equation}\label{Riesz9}
E=Ker(L^r) \dot{+} Ran(L^r).
\end{equation}
In view of the closedness of the range $L^r$ and \cite[Theorem 3.2]{LAN}, we have
\begin{equation}\label{Riesz10}
E=Ker(L^{*r}) \oplus Ran(L^r) = Ker(L^r) \oplus Ran(L^{*r}).
\end{equation}
According to (\ref{Riesz9}) and (\ref{Riesz10}) we infer, $Ran(L^{*r})=Ran(L^{r})$, that is $L^r$
is an EP operator with closed range.
\end{proof}

\begin{corollary} \label{RieszKEP}
Let $E$ be a Hilbert C*-module over a unital C*-algebra $ \mathcal{A}$ that satisfies
property $\mathbb{[H]}$ and $C \in \mathcal{K}(E)$. If $L=I-C$, then the EP operator $L^r$
admits the matrix form
\begin{equation*}
L^r=\begin{bmatrix} I_{r} & 0 \\ 0 & 0 \end{bmatrix}:
\begin{bmatrix} Ran(L^r) \\ Ker(L^r) \end{bmatrix} \to
\begin{bmatrix} Ran(L^r)  \\ Ker(L^r)  \end{bmatrix},
\end{equation*}
where, $I_{r}$ is an invertible in $\mathcal{L}(Ran(L^r),Ran(L^r))$. Indeed, $I_{r}$ is the identity operator on $Ran(L^r)$.
\end{corollary}
\begin{proof}
The matrix form follows from the closedness of the range of $L^r$, Theorem \ref{RieszK3}, and Lemma 3.6 of \cite{SHA/EP}.
\end{proof}

In the following two corollaries, we present a reformulation of the classical results of Riesz
and Fredholm \cite{Fredholm,Riesz} for Hilbert modules with property $\mathbb{[H]}$, separately analyzing the cases $r=0$ and $r>0$.

\begin{corollary} \label{RieszK4}
Let $E$ be a Hilbert C*-module over a unital C*-algebra $ \mathcal{A}$ that satisfies
property $\mathbb{[H]}$ and $C \in \mathcal{K}(E)$. Let $L=I-C$ is injective, then:
\begin{enumerate}
  \item The modular operator $L$ is surjective.
  \item The homogeneous equation $x-Cx=0$ only has the trivial solution $x=0$, and
  for all $f \in E$, the inhomogeneous equation
  $$x-Cx=f$$
has a unique solution $x\in E$. Moreover, this solution depends continuously on $f$.
\end{enumerate}
\end{corollary}
\begin{proof}
In view of Theorem \ref{RieszK2} and the fact that $Ker(L)=\{0 \}$, we infer $r=0$ and $Ran(L)=E$, that is,
$L$ is surjective.

To prove the second part, we have to show that the inverse of $L=I-C$ is bounded, i.e. $(I-C)^{-1} \in \mathcal{L}(E)$.
Suppose $L^{-1}$ is not bounded. Then there exists a sequence $(y_n)$ with $ \| y_n \| = 1 $
such that the sequence $x_n := L^{-1} y_n $ is not bounded. We set
$ u_n := \frac{y_n}{\|x_n\|}$ and $z_n := \frac{ x_n}{\| x_n\|}$.
Then $ \|z_n\| = 1$, $Lz_n= \frac{Lx_n}{ \|x_n \|}=u_n$ and $u_n \to 0$ as $ n \to \infty $.
Utilizing the property $\mathbb{[H]}$, Lemma \ref{Arambasic1} and  compactness of $C$,
we can find a subsequence $z_{n_k}$ such that $Cz_{n_k} \to z \in E$ as $k \to \infty $. Then, the equalities
$$ Lz_{n_k} = z_{n_k} -Cz_{n_k}=u_{n_k},$$
imply that $z_{n_k} \to z$ with $Lz=0$. Thus $z \in Ker(L)=\{0 \}$, which contradicts $ \| z_n\| = 1$  for all $n \in \mathbb{N}$.
Therefore,  $(I-C)^{-1}f$ is the unique solution of the equation $x-Cx=f$, provided that $(I-C)^{-1}$ exists in $\mathcal{L}(E)$.
\end{proof}

\begin{corollary} \label{RieszK5}
Let $E$ be a Hilbert C*-module over a unital C*-algebra $ \mathcal{A}$ that satisfies
property $\mathbb{[H]}$ and $C \in \mathcal{K}(E)$. Let $L=I-C$ is not injective, then:
\begin{enumerate}
  \item The $ \mathcal{A}$-submodule $Ker(L)$ is
  finitely generated and the range $Ran(L)$ is a proper closed $ \mathcal{A}$-submodule.
  \item If the homogeneous equation $x-Cx=0$ has a nontrivial solution, then the inhomogeneous equation
  $$x-Cx=f$$
  is either unsolvable or its general solution is of the form
  $$ x= x_0 + \sum_{k=1}^{k=m} a_k x_k,$$
where $x_1, \dots, x_m $ are the generators of finitely generated $ \mathcal{A}$-submodule $Ker(L)$,
$a_1, \dots, a_m$ are in $ \mathcal{A}$,
and $x_0$ denotes a special solution of the inhomogeneous equation.
\end{enumerate}
\end{corollary}
\begin{proof}
The $ \mathcal{A}$-submodule $Ker(L)$ is finitely generated by Lemma \ref{RieszK1}.
Since $L=I-C$ is not injective, $Ker(L) \neq \{ 0 \}$ and $r>0$, and so
$Ran(L) \neq E$ by Theorem \ref{RieszK2}.

To prove the second part, suppose $y_0$ is a nontrivial solution of the homogeneous equation $x-Cx=0$, and
the inhomogeneous equation $x-Cx=f$ is solvable with a solution $x_0$. Then, the general solution of $x-Cx=f$ is of the form
$x=x_0+y_0$. Since the $ \mathcal{A}$-submodule $Ker(L)$ is finitely generated, we can express $ y_0$ as
$$ y_0= \sum_{k=1}^{k=m} a_k x_k,$$ where $x_1, \dots, x_m $ are the generators of $ \mathcal{A}$-submodule $Ker(L)$, and
$a_1, \dots, a_m$ are elements of $ \mathcal{A}$. This completes the proof of (2).

\end{proof}

\begin{theorem} \label{RieszK6}
Let $E$ be a Hilbert C*-module over a unital C*-algebra $ \mathcal{A}$ that satisfies
property $\mathbb{[H]}$. Let $C \in \mathcal{K}(E)$ and $L=I-C$.
\begin{enumerate}
  \item The kernel of $L^r$ is a finitely generated submodule of $E$ and the orthogonal projection
  $$P : E =Ker(L^{r}) \oplus Ran(L^r) \to Ker(L^r)$$ is a compact projection.
  \item The operator $I- C - P$ is a bijection.
\end{enumerate}
\end{theorem}

\begin{proof}
The operator $L^r$ can be expressed as $L^r=I - C_r$, where $C_r= \sum_{j=1}^{r} (-1)^{j-1} \, \frac{r!}{j!(r-j)!} \, C^j$ is
a compact operator on $E$.
According to Lemma \ref{RieszK1} and Lemma \ref{Wegge257}, the kernel of $L^r$ is a finitely generated submodule of $E$. If
$E_0=Ran(P)=Ker(L^r)$, then $\mathcal{K}(E_0)$ is a unital C*-algebra. Then the identity map on the Hilbert $ \mathcal{A}$-module
$E_0$ can be written as $ \sum _{i=1}^{i=m} \theta_{x_i , \, y_i}$, for some $x_i, \, y_i \in E_0$, and so we have
$$Px=I_{ \mathcal{K}(E_0) }(Px) = \sum _{i=1}^{i=m} \theta_{x_i , y_i} (Px) = \sum _{i=1}^{i=m} x_i \, \langle y_i, Px  \rangle =
 \sum _{i=1}^{i=m} \theta_{x_i , Py_i}(x), $$
for all $x \in E$. Consequently, $P$ is a compact projection in $ \mathcal{K}(E)$.

To prove the second part, we need to show that $Ker( I - C- P)=\{0 \}$. Let $x \in Ker(L-P)$, then
$$L^{r+1}x + 0=L^{r+1}x - L^r(Px)= L^r(Lx-Px)=L^r(0)=0.$$
That is, $x \in Ker(L^{r+1})= Ker(L^{r})$. Hence, $Px=x$, which implies $0=Lx - Px = Lx-x$,  or equivalently, $x=Lx$.
It follows that $x$ is a fixed point for all powers of $L$, and therefore
$$x= L^r x =0.$$
This proves that $Ker(L-P)=\{0 \}$. Utilizing Corollary \ref{RieszK4} for the compact operator $C+P$,  we
conclude that $L-P=I-(C+P)$ is surjective.
\end{proof}

We end our paper with the remark that the operator $L=I-C$ in our results can be generalized to $L=\lambda I -C$,
where $ \lambda \in \mathbb{C}$ and $C \in \mathcal{K}(E)$.

\subsection*{Funding} This research did not receive any specific grant from funding agencies
in the public, commercial, or not-for-profit sectors.
\subsection*{Data Availability}
Data sharing not applicable to this article as no datasets were generated
or analysed during the current study.
\subsection*{Conflict of interest}
On behalf of all authors, the corresponding author states that there is
no conflict of interest.
%\subsection*{Acknowledgment}

%%%%%%%%%%%%%%%%%%%%%%%%%%%%%%%%%%%%%%%%%%%%%%%%%%%%%%%%%%%%%%%%%%%%%%%%%%%%%%%%%%%%%%%%%%%%%%%%%%%%%%%%%%%%%%%%%%%%%

\end{document}